\documentclass[12pt]{amsart}
\usepackage{hyperref}

\usepackage{ amssymb }

\numberwithin{equation}{section}

\usepackage{hyperref}

\usepackage{float}

\usepackage{pgf,tikz}

\usepackage{enumitem}

\newcommand \reg{\operatorname{reg}}

\newcommand\link{\operatorname{link}}

\newcommand \cf{\mathsf{CF}}

\newcommand \K{\mathbb{K}}

\newcommand \cl{\operatorname{Cl}}

\newcommand \CH{\mathcal{CH}}

\newcommand \ind{\operatorname{Ind}}

\newtheorem{theorem}{Theorem}[section]

\newtheorem{definition}[theorem]{Definition}
\newtheorem{cons}[theorem]{Construction}
\newtheorem{lemma}[theorem]{Lemma}
\newtheorem{proposition}[theorem]{Proposition}
\newtheorem{example}[theorem]{Example}

\newtheorem{remark}[theorem]{Remark}

\newtheorem{corollary}[theorem]{Corollary}

\newtheorem*{notation*}{Notation}

\begin{document}

\title[Shellability in Clique-Free Complexes]
{Shellability in Clique-Free Complexes of Graphs}

\author{Rakesh Ghosh}
\email{rakeshghosh1591@gmail.com}

\author{S~Selvaraja}
\address{Department of Mathematics, Indian Institute of Technology Bhubaneswar,
Bhubaneswar 752050, India}
\email{selvas@iitbbs.ac.in}



\begin{abstract}
We study combinatorial and algebraic properties of $t$-clique-free
complexes, a family of simplicial complexes associated with finite simple
graphs that generalize the classical independence complex. For a graph $G$ and
an integer $t \ge 2$, the $t$-clique-free complex $\cf_t(G)$ is the simplicial
complex on the vertex set of $G$ whose faces are the subsets inducing no
cliques of size $t$.

Our main results provide sufficient conditions for shellability and related
decomposability properties of $t$-clique-free complexes. In particular, we show
that if $G$ is a $t$-diamond-free chordal graph (in particular, a block graph), then $\cf_t(G)$ is
$(t-2)$-decomposable and hence shellable. We also investigate how graph
modifications via clique attachments influence shellability. Generalizing
earlier constructions involving whiskers and clique extensions, we introduce
the following operation: given a graph $H$, a subset $S \subseteq V(H)$, and an
integer $t \ge 2$, we form a graph $\cl(H,S,t)$ by attaching to each vertex in
$S$ a clique of size at least $t$. We prove that $\cf_t(H \setminus S)$ is
shellable if and only if $\cf_t(\cl(H,S,t))$ is shellable. This yields a
flexible method for constructing shellable complexes, particularly when $S$ is
a cycle cover. In addition, we extend the notion of clique whiskering and show
that for any graph admitting a clique vertex-partition, the resulting
$t$-clique whiskering produces a pure and shellable, and hence
Cohen-Macaulay, $t$-clique-free complex.

Finally, we establish a Fr\"oberg-type result linking chordality and linear
resolutions. We show that for any chordal graph $G$, the edge ideal of the
complement $t$-clique clutter $\overline{\CH_t(G)}$ admits a $t$-linear
resolution over any field.
\end{abstract}

\maketitle
\section{Introduction}

In this paper, we study a family of simplicial complexes arising from graphs that
are defined by forbidding cliques of a fixed size. These complexes, which we
refer to as clique-free complexes, arise naturally as the Stanley-Reisner
complexes of the well-established \emph{$t$-clique ideals} of a graph. More
precisely, for a graph $G = (V(G), E(G))$ and an integer $t \ge 2$, the
\textit{clique-free complex} $\cf_t(G)$ is the simplicial complex on $V(G)$ whose
faces are the subsets $F \subseteq V(G)$ that contain no $t$-clique. Equivalently,
$\cf_t(G)$ is the Stanley-Reisner complex of the $t$-clique ideal $I_t(G)$. When
$t = 2$, this construction recovers the classical \textit{independence complex}
$\ind(G)$, whose faces are the independent sets of $G$. Thus, the family
$\{\cf_t(G)\}_{t \ge 2}$ forms a natural filtration
\[
\cf_2(G) \subseteq \cf_3(G) \subseteq \cf_4(G) \subseteq \cdots .
\]

Clique-free complexes and their associated ideals, including $t$-clique ideals
and the dual notion of $t$-independence ideals, have been introduced and studied
in prior work in combinatorial commutative algebra and topological combinatorics
(see, for example, \cite{KAK26, HKMR20, KM19, Mor18}). In particular, properties
such as shellability, vertex decomposability, and Cohen-Macaulayness have been
investigated for specific graph families. The purpose of the present work is to
contribute further to this line of research by establishing new, general
sufficient conditions on the graph $G$ that guarantee the complex $\cf_t(G)$
exhibits these desirable combinatorial and algebraic properties.

Understanding when simplicial complexes associated with graphs satisfy
properties such as vertex decomposability, shellability, or (sequential)
Cohen-Macaulayness is a central theme in topological combinatorics. These
properties are related by the following well-known hierarchy:
\[
\begin{aligned}
&\text{vertex-decomposable} \implies \text{shellable} \implies
\text{sequentially Cohen-Macaulay}, \\
&\left.
    \begin{aligned}
        &\text{pure vertex-decomposable}, \\
        &\text{pure shellable}
    \end{aligned}
    \right\}
    \implies \text{Cohen-Macaulay}.
\end{aligned}
\]
Characterizing simplicial complexes that belong to these classes has been an
active area of research (see, for example,
\cite{crupi, fv, VanVilla, vill_cohen, Wood2, Russ11}).

Our investigation is motivated by a fundamental result of
Dochtermann-Engstr\"om~\cite[Theorem~4.1]{DochEng09} and
Woodroofe~\cite[Corollary~7]{Wood2}, who proved that the independence complex
$\ind(G)$ of any chordal graph $G$ is vertex decomposable, and hence shellable.
This result highlights the strong interplay between graph-theoretic structure
and the topological properties of associated simplicial complexes, and naturally
motivates the study of higher-order generalizations of independence complexes.
Previous studies have examined the complexes $\cf_t(G)$ for special classes of
graphs. For instance, Moradi~\cite{Mor18} studied purity and shellability
properties of the $t$-clique-free complex of the complement of a cycle,
$\cf_t(\overline{C_n})$, and investigated the Alexander dual of
$\cf_t(\overline{G})$, where $\overline{G}$ denotes the complement of the graph
$G$. Furthermore, Khosh-Ahang and Moradi~\cite{KM19} established vertex
decomposability of $\cf_t(\overline{G})$ for several well-known graph classes,
including path graphs, star graphs, double star graphs, trees of diameter less
than four, broom graphs, and double broom graphs.

In the present work, we move beyond these specific families and provide broader,
systematic sufficient conditions on the graph $G$ that ensure the complex
$\cf_t(G)$ is shellable. Our results unify and extend several earlier statements
and provide new tools for analyzing the combinatorial topology of clique-free
complexes for arbitrary graphs.

Our investigation also reveals that the situation for $t \ge 3$ is more subtle
than in the classical case $t = 2$. As demonstrated in
Example~\ref{not-shellable}, there exist chordal graphs $G$ for which $\cf_t(G)$
fails to be sequentially Cohen-Macaulay, and consequently is neither shellable
nor vertex decomposable. This observation motivates the following central
question.
\textit{Which natural subclasses of chordal graphs guarantee that $\cf_t(G)$ is
vertex decomposable, shellable, or sequentially Cohen-Macaulay for all
$t \ge 3$?}

To address this question, we employ the framework of \emph{$k$-decomposability}, a
hierarchical generalization of vertex decomposability. For a (not necessarily
pure) simplicial complex of dimension $k$, the following implications hold (see,
for example, \cite{J05}):
\[
\text{$0$-decomposable} \implies \text{$1$-decomposable} \implies \cdots \implies
\text{$k$-decomposable} \iff \text{shellable},
\]
where $0$-decomposable is equivalent to vertex decomposable. This concept,
originally introduced for pure complexes by Provan and Billera~\cite{ProvLouis}
and later extended to the non-pure setting by Bj\"orner-Wachs~\cite{BjWachs} and
Woodroofe~\cite{Russ11}, provides a graded and flexible tool for establishing
shellability.

Our main result provides a sufficient condition for shellability in terms of a
forbidden graph substructure.

\medskip
\noindent\textbf{Theorem~\ref{main}.}\emph{
Let $G$ be a $t$-diamond-free chordal graph. Then, for all $t \ge 3$, the
complex $\cf_t(G)$ is $(t-2)$-decomposable. In particular, $\cf_t(G)$ is shellable
for all $t \ge 3$.}

\medskip
Since block graphs form a subclass of $t$-diamond-free chordal graphs,
Theorem~\ref{main} applies directly to this class, guaranteeing the shellability
of $\cf_t(G)$ for all $t \ge 3$.

We now shift our focus to the following natural question: given an arbitrary
graph $G$, how can one modify $G$ so that the associated clique-free complex
$\cf_t(G)$ is vertex decomposable, shellable, or sequentially
Cohen-Macaulay for all $t \ge 2$? A foundational result of
Villarreal~\cite{vill_cohen} shows that, for any graph $G$, the independence
complex of its whiskered version $W(G)$ - formed by attaching a pendant vertex to
each vertex of $G$ - is Cohen-Macaulay. This was subsequently generalized by
Francisco and H\'a~\cite{FH}, who proved that if $G \setminus S$ is either chordal
or a $5$-cycle $C_5$, then the independence complex of $G \cup W(S)$, where only
the vertices in $S \subseteq V(G)$ are whiskered, is sequentially
Cohen-Macaulay. Further extending this line of research, Cook~II and
Nagel~\cite{CookNagel} introduced the notion of vertex clique-whiskered graphs
$G^{\pi}$ and showed that $\ind(G^{\pi})$ is both pure and vertex decomposable.

Motivated by these results, we investigate how attaching cliques of varying
sizes to a graph affects the shellability of the clique-free complex
$\cf_t(G)$. Our approach is based on the following construction: given a graph
$H$, a subset $S \subseteq V(H)$, and an integer $t \ge 2$, we form a new graph
$\cl(H,S,t)$ by attaching to each vertex $v \in S$ a new clique $K_v$ of size at
least $t$ containing $v$.

Our main result shows that this construction preserves shellability.

\vskip 1mm
\noindent
\textbf{Theorem~\ref{she-extension}.}\emph{
Let $G = \cl(H,S,t)$ with $t \ge 2$. Then the complex $\cf_t(H \setminus S)$ is
shellable if and only if $\cf_t(G)$ is shellable.}

Theorem~\ref{she-extension} yields several interesting consequences. For
instance, Corollary~\ref{cor:clique-ext} shows that if $S \subseteq V(H)$ is a
cycle cover, and hence a vertex cover, of $H$, then the clique-free complex
$\cf_t(G)$ is shellable for $G = \cl(H,S,t)$. This provides a general method for
constructing shellable clique-free complexes for all $t \ge 2$ by
strategically adding cliques. Moreover, our framework recovers and unifies key
results of Francisco-H\'a~\cite{FH} and Van~Tuyl-Villarreal~\cite{VanVilla}, as
summarized in Corollary~\ref{cor:FH-VanVilla}.

In related work, Cook~II and Nagel~\cite{CookNagel} introduced the notion of
clique whiskering for graphs. We generalize this construction (see
Construction~\ref{CN}) and establish the following results.

\begin{theorem}\label{main-thm}
Let $G$ be a graph with a clique vertex-partition
$\Pi = \{W_1, \ldots, W_p\}$. For any integer $t \ge 2$, the $t$-clique whiskered
graph $G(\Pi,t)$ satisfies:
\begin{enumerate}
    \item[(i)] The complex $\cf_t(G(\Pi,t))$ is pure
    (Theorem~\ref{pure}).
    \item[(ii)] The complex $\cf_t(G(\Pi,t))$ is shellable, and hence
    Cohen-Macaulay (Theorem~\ref{clique-vertex}).
\end{enumerate}
\end{theorem}

Let $G$ be a graph with vertex set $V(G)=\{x_1,\dots,x_n\}$. The \emph{edge ideal}
of $G$, introduced by Villarreal~\cite{vill_cohen}, is the squarefree monomial
ideal
$I(G)=(x_i x_j \mid \{x_i,x_j\}\in E(G))$
in the polynomial ring $\mathbb{K}[x_1,\dots,x_n]$, where $\mathbb{K}$ is a
field. Edge ideals provide a fundamental bridge between graph theory and
commutative algebra, allowing combinatorial properties of a graph to be studied
through homological invariants of its associated ideal.

A cornerstone result in this area is Fr\"oberg’s theorem~\cite{froberg}, which
characterizes graphs whose edge ideals have a linear resolution. Specifically,
it states that the edge ideal $I(\overline{G})$ of the complement graph
$\overline{G}$ has a linear resolution if and only if $G$ is chordal. This result
reveals a deep connection between chordality and the homological behavior of
edge ideals. It is worth noting that for quadratic squarefree monomial ideals,
that is, edge ideals of graphs, the existence of a linear resolution depends
only on the combinatorial structure of the graph and is independent of the
choice of the field $\mathbb{K}$. However, this invariance fails for squarefree
monomial ideals generated in higher degrees; as shown in~\cite[Section~4]{kat},
the resolution properties in this setting may depend on the characteristic of
the field.

In order to extend Fr\"oberg’s theorem beyond the quadratic case, it is natural
to work in the more general setting of \emph{clutters} (also called
\emph{Sperner families}), which generalize graphs by allowing edges of arbitrary
cardinality. Formally, a clutter $\mathcal{C}$ on a vertex set
$V(\mathcal{C})=\{x_1,\dots,x_n\}$ is a collection of subsets of $V(\mathcal{C})$,
called \emph{edges} or \emph{circuits}, such that no edge is properly contained
in another. If every edge of $\mathcal{C}$ has cardinality $d$, then
$\mathcal{C}$ is called a \emph{$d$-uniform clutter} (so graphs are precisely
$2$-uniform clutters). The \emph{complement clutter}
$\overline{\mathcal{C}}$ consists of all $d$-subsets of $V(\mathcal{C})$ that do
not belong to $\mathcal{C}$. The \emph{edge ideal} of a clutter $\mathcal{C}$ is
the squarefree monomial ideal
\[
I(\mathcal{C})=\Bigl(\prod_{x_i\in e} x_i \,\Bigm|\, e\in E(\mathcal{C})\Bigr),
\]
which generalizes the edge ideal of a graph and plays a central role in the
study of higher-degree squarefree monomial ideals; see, for example,
\cite{ha_adam}.

Motivated by Fr\"oberg’s theorem, we associate to any graph $G$ the
\emph{$t$-clique clutter} $\mathcal{CH}_t(G)$, whose edges are the $t$-cliques of
$G$. The edge ideal of $\mathcal{CH}_t(G)$ is precisely the Stanley-Reisner
ideal of the $t$-clique-free complex $\cf_t(G)$. Our main result in this
direction extends Fr\"oberg’s theorem to higher degrees.

\vskip 1mm
\noindent
\textbf{Theorem~\ref{linear-rs}.} \emph{
Let $G$ be a chordal graph. Then, for every integer $t\ge 2$, the edge ideal
$I(\overline{\mathcal{CH}_t(G)})$ has a $t$-linear resolution over any field
$\mathbb{K}$.}

The paper is organized as follows.
Section~\ref{preliminaries} introduces the necessary background and notation,
including simplicial complexes with emphasis on shellability and
Cohen-Macaulay properties, as well as edge ideals and linear resolutions.
In Section~\ref{shellability}, we establish our central combinatorial result,
showing that for any $t$-diamond-free chordal graph $G$, the clique-free complex
$\cf_t(G)$ is $(t-2)$-decomposable (Theorem~\ref{main}).
Section~\ref{constructions} develops graph modification techniques that preserve
shellability, focusing in particular on constructions obtained by attaching
$t$-cliques to specified vertex subsets, with applications to cycle covers and
clique-whiskered graphs.
Finally, the algebraic aspect of the theory is addressed in
Section~\ref{algebra}, where we extend Fr\"oberg’s theorem by proving that the
edge ideal of the complement $t$-clique clutter of any chordal graph has a
$t$-linear resolution (Theorem~\ref{linear-rs}).

\section{Notation and Preliminaries}\label{preliminaries}

We begin by recalling essential concepts from commutative algebra and algebraic combinatorics. Let $R = \K[x_1, \dots, x_n]$ denote the polynomial ring over a field $\K$ with standard grading, where each variable has degree one. For any homogeneous ideal $I \subseteq R$, its \emph{minimal graded free resolution} takes the form of an exact sequence
\[
0 \to \bigoplus_{j \in \mathbb{N}} R(-j)^{\beta_{p,j}(I)} \to \cdots \to \bigoplus_{j \in \mathbb{N}} R(-j)^{\beta_{0,j}(I)} \to I \to 0,
\]
where $p \leq n$ is the projective dimension of $I$, and $R(-j)$ denotes the graded twist of $R$ satisfying $R(-j)_d = R_{d-j}$. The exponents $\beta_{i,j}(I)$, known as the \emph{graded Betti numbers}, record the number of degree-$j$ generators required for the $i$-th syzygy module. 

The \emph{Castelnuovo-Mumford regularity} of $I$, denoted $\reg(I)$, measures the complexity of the resolution and is defined as $\reg(I) = \max\{j - i \mid \beta_{i,j}(I) \neq 0\}$. A homogeneous ideal $I$ is said to have a \emph{$d$-linear resolution} if it satisfies two conditions: (1) all minimal generators share the same degree $d$, and (2) $\beta_{i,i+j}(I) = 0$ for all $i \geq 1$ and $j \neq d$. Equivalently, the resolution simplifies to 
\[
0 \to R(-d - p)^{\beta_{p,p+d}(I)} \to \cdots \to R(-d)^{\beta_{0,d}(I)} \to I \to 0,
\]
in which case the regularity equals the generating degree, $\reg(I) = d$.

A \emph{simplicial complex} $\Delta$ on a vertex set $V = \{x_1, \dots, x_n\}$ is a family of subsets of $V$ (called \emph{faces}) such that:
\begin{enumerate}[label=(\roman*)]
    \item Every singleton $\{x_i\}$ is a face, and
    \item $\Delta$ is closed under inclusion (i.e., if $F \in \Delta$ and $F' \subseteq F$, then $F' \in \Delta$).
\end{enumerate}
The maximal faces under inclusion are called \emph{facets}. The \emph{dimension} of a face $F \in \Delta$ is $\dim F = |F| - 1$, and the \emph{dimension} of $\Delta$ is $\dim \Delta = \max\{ \dim F \mid F \in \Delta \}$. If all facets of $\Delta$ have the same dimension, we say that $\Delta$ is \emph{pure}. If the simplicial complex has exactly one facet, we call it a \emph{simplex}.
For any face $F \in \Delta$, we define:
 the \emph{link} of $F$ as $\operatorname{link}_\Delta(F) = \{ F' \subseteq V \mid F' \cap F = \emptyset \text{ and } F' \cup F \in \Delta \}$, and
     the \emph{deletion} of $F$ as $\Delta \setminus F = \{ H \in \Delta \mid H \cap F = \emptyset \}$.

The \emph{Alexander dual} of a simplicial complex $\Delta$ is defined as $\Delta^\vee = \{V \setminus F \mid F \notin \Delta\}$.
The \emph{join} of two simplicial complexes $\Delta_1$ and $\Delta_2$ on disjoint vertex sets $V_1$ and $V_2$ is the simplicial complex $\Delta_1 * \Delta_2$ on the vertex set $V_1 \cup V_2$, whose faces are given by
$\{\sigma_1 \cup \sigma_2 \mid \sigma_1 \in \Delta_1, \ \sigma_2 \in \Delta_2\}.$
Note that $\Delta * \{\emptyset\} = \Delta$.
The \emph{Stanley-Reisner ideal} of $\Delta$, denoted $I_\Delta$, is the squarefree monomial ideal in the polynomial ring $\K[x_1, \ldots, x_n]$ defined by
$I_\Delta = \big( x_{i_1}  \cdots x_{i_s} \mid \{x_{i_1},  \dots, x_{i_s}\} \notin \Delta \big).$
In other words, $I_\Delta$ is generated by the minimal non-faces of $\Delta$.

The notion of a \emph{shedding face} was first introduced by Jonsson in \cite{J05}:

\begin{definition}[\cite{J05}, Definition 2.10]
A face $\sigma \in \Delta$ is called a \emph{shedding face} if for every $\tau \in \operatorname{star}(\Delta, \sigma) = \{ \rho \in \Delta \mid \sigma \subseteq \rho \}$ and every $v \in \sigma$, there exists a vertex $w \in V \setminus \tau$ such that
$(\tau \cup \{w\}) \setminus \{v\} \in \Delta.$
\end{definition}

When $\sigma$ is a single vertex, this definition coincides with the notion of a shedding vertex as defined by Björner and Wachs \cite[Section 11]{BjWachs}. Furthermore, when $\Delta$ is pure, the definition aligns with that of Provan and Billera \cite[Definition 2.1]{ProvLouis}.

\begin{definition}[$t$-Decomposable Complex]
A simplicial complex $\Delta$ is said to be \emph{$t$-decomposable} (for $t \geq -1$) if one of the following holds:
\begin{enumerate}
    \item $\Delta$ is a simplex (including the void complex $\emptyset$ and the irrelevant complex $\{\emptyset\}$); or
    \item There exists a shedding face $\sigma \in \Delta$ with $\dim(\sigma) \leq t$ such that:
    \begin{itemize}
        \item[(i)] the deletion $\Delta \setminus \sigma$ is $t$-decomposable, and
        \item[(ii)] the link $\operatorname{link}_\Delta(\sigma)$ is $t$-decomposable.
    \end{itemize}
\end{enumerate}
\end{definition}

\begin{remark}
\leavevmode
\begin{itemize}
    \item[(i)] By definition, both the void complex $\emptyset$ and the irrelevant complex $\{\emptyset\}$ are $t$-decomposable for all $t \geq -1$.
    \item[(ii)] When $t = 0$, the notion of $t$-decomposability coincides with that of a \emph{vertex-decomposable} complex.
    \item[(iii)] For $t \geq 0$, $t$-decomposability is monotone in $t$:
    \[
    t\text{-decomposable} \quad \Rightarrow \quad (t+1)\text{-decomposable}.
    \]
\end{itemize}
\end{remark}

\begin{definition}
    A simplicial complex $\Delta$ is \emph{shellable} if its facets admit a linear order $F_1, \ldots, F_s$ (called a \emph{shelling order}) satisfying the following condition:
For every pair of indices $i < j$, there exists a vertex $v \in F_j \setminus F_i$ and an index $\ell < j$ such that
$F_j \setminus F_\ell = \{v\}.$
\end{definition}

The following result establishes a characterization of shellability in terms of decomposability:
\begin{theorem}\cite{J05}
A $d$-dimensional simplicial complex $\Delta$ is shellable if and only if it is $d$-decomposable.
\end{theorem}

Let $R = \K[x_1,\ldots,x_n].$
A graded R-module M is called is \emph{sequentially Cohen-Macaulay} (over $\K$) if there exists a filtration
\[
0 = M_0 \subsetneq M_1 \subsetneq \cdots \subsetneq M_r = M
\]
of graded $R$-submodules satisfying:
\begin{enumerate}
    \item[(i)] Each subquotient $M_i/M_{i-1}$ is Cohen-Macaulay;
    \item[(ii)] The sequence of Krull dimensions is strictly increasing:
    \[
    \dim(M_i/M_{i-1}) < \dim(M_{i+1}/M_i) \quad \text{for all } 1 \leq i < r.
    \]
\end{enumerate}

A simplicial complex $\Delta$ is \emph{(sequentially) Cohen-Macaulay over $\K$} if its Stanley-Reisner ring $\K[x_1,\ldots,x_n]/I_{\Delta}$ is (sequentially) Cohen-Macaulay as a $\K[x_1,\ldots,x_n]$-module.

A \emph{clutter} $\mathcal{C}$ on a vertex set $V(\mathcal{C}) = \{x_1, \dots, x_n\}$ is a collection of subsets of $V(\mathcal{C})$, called \emph{circuits} or \emph{edges}, where no two distinct circuits $e_1$ and $e_2$ satisfy $e_1 \subseteq e_2$. This means the circuits are inclusion-wise incomparable. When every circuit contains exactly $d$ vertices, we call $\mathcal{C}$ a \emph{$d$-uniform clutter}, and each circuit is referred to as a \emph{$d$-circuit}. Note that, any simple graph corresponds naturally to a $2$-uniform clutter, where its edges become the circuits of the clutter.
For any $d$-uniform clutter $\mathcal{C}$, we define its \emph{complement clutter} $\overline{\mathcal{C}}$ on the same vertex set $V(\mathcal{C})$. This complement consists of all $d$-element subsets of $V(\mathcal{C})$ that are not circuits of $\mathcal{C}$. Formally,
$\overline{\mathcal{C}} = \big\{ e \subseteq V(\mathcal{C}) \mid |e| = d \text{ and } e \notin \mathcal{C} \big\}.$

Let $\mathcal{C}$ be a clutter. A subset $I \subseteq V(\mathcal{C})$ is called an \emph{independent set} of $\mathcal{C}$ if it contains no circuits. 
The \emph{contraction} $\mathcal{C}/v$ is the clutter on the vertex set $V(\mathcal{C}) \setminus \{v\}$, whose edges are the minimal sets among 
$\left\{ e \setminus \{v\} \mid e \text{ is an edge of } \mathcal{C} \right\}.$
The \emph{independence complex} $\ind(\mathcal{C})$ is the simplicial complex consisting of all independent sets of $\mathcal{C}$. 

Given a clutter $\mathcal{C}$ with vertex set $V(\mathcal{C}) = \{x_1, \dots, x_n\}$, we identify each vertex $x_i \in V(\mathcal{C})$ with a variable $x_i$ in the polynomial ring $\mathbb{K}[x_1, \dots, x_n]$. The \emph{edge ideal} of $\mathcal{C}$, denoted $I(\mathcal{C})$, is the squarefree monomial ideal in $\mathbb{K}[x_1, \dots, x_n]$ generated by monomials corresponding to the edges of $\mathcal{C}$:
$I(\mathcal{C}) = \left( \prod_{x_i \in e} x_i \,\middle|\, e \in E(\mathcal{C}) \right).$
This construction generalizes the edge ideal of a graph and plays a central role in the study of squarefree monomial ideals of higher degree~\cite{ha_adam}.

Let $G$ be a graph without isolated vertices. We denote by $V(G)$ and $E(G)$ the vertex set and edge set of $G$, respectively. The \emph{degree} of a vertex $x \in V(G)$, denoted $\deg_G(x)$ or simply $\deg(x)$, is the number of edges incident to $x$. A subgraph $H \subseteq G$ is called \emph{induced} if for any two vertices $u, v \in V(H)$, the edge $\{u, v\}$ belongs to $E(H)$ if and only if $\{u, v\} \in E(G)$. For a subset $A \subseteq V(G)$, the \emph{induced subgraph} on $A$ is denoted $G[A]$.
The \emph{neighborhood} of a set of vertices $\{u_1, \ldots, u_r\} \subseteq V(G)$ is defined as
\[ N_G(u_1, \ldots, u_r) = \big\{v \in V(G) \mid \{u_i, v\} \in E(G) \text{ for some } 1 \leq i \leq r\big\}. \]
We also write $N_G[u_1, \ldots, u_r] = N_G(u_1, \ldots, u_r) \cup \{u_1, \ldots, u_r\}$.
For a subset $U \subseteq V(G)$, we write $G \setminus U$ for the induced subgraph on $V(G) \setminus U$.
A graph $G$ is called a \emph{complete graph} or a \emph{clique} if every pair of distinct vertices in $G$ is connected by an edge. A vertex $v \in V(G)$ is called a \emph{simplicial vertex} if its neighborhood $N_G(v)$ forms a clique in $G$.

A graph $G$ is \emph{chordal} if every induced cycle in $G$ has length exactly~3. Chordal graphs admit a well-known characterization in terms of simplicial vertices. In their seminal work, \cite{Dirac61} proved the following fundamental result: every chordal graph contains at least one simplicial vertex.

A subset $A \subseteq V(G)$ is said to be \emph{$t$-clique-free} if the induced subgraph $G[A]$ contains no clique of size $t$. The \emph{$t$-clique-free complex}, denoted $\cf_t(G)$, is the simplicial complex whose faces are all $t$-clique-free subsets of $V(G)$. When $t = 2$, this reduces to the independence complex: $\ind(G) = \cf_2(G)$, making the $t$-clique-free complex a natural generalization of the independence complex.

\begin{definition}[$t$-Clique Clutter]
The $t$-clique clutter of a graph $G$, denoted $\CH_t(G)$, is the clutter over the vertex set $V(G)$ whose circuits are precisely the $t$-vertex cliques (or $t$-cliques) of $G$. Formally,
\[
\CH_t(G) = \big\{ C \subseteq V(G) \mid |C| = t \text{ and } G[C] \text{ is a complete graph} \big\}.
\]

\end{definition}
The study of clique clutters and their independence complexes provides a natural framework for understanding the structure of cliques in graphs. The following proposition establishes a fundamental connection between the independence complex of the $t$-clique clutter and the $t$-clique-free complex, highlighting how these objects encode the same combinatorial information about the graph.

\begin{proposition}\label{ind}
For any graph $G$ and integer $t \geq 2$, the independence complex of the $t$-clique clutter $\mathcal{CH}_t(G)$ coincides with the $t$-clique-free complex of $G$. That is,
$\ind\big(\mathcal{CH}_t(G)\big) = \cf_t(G).$
\end{proposition}

\begin{proof}
By definition, a subset $A \subseteq V(G)$ is independent in $\mathcal{CH}_t(G)$ if and only if it contains no $t$-clique of $G$, which is precisely the condition for $A$ to be a face of $\cf_t(G)$.
\end{proof}

\section{Shellability of Clique-Free Complexes for Diamond-Free Chordal Graphs}
\label{shellability}

In this section, we present our main result (Theorem~\ref{main}), which shows
that for any $t$-diamond-free chordal graph $G$ and every integer $t \ge 3$, the
clique-free complex $\cf_t(G)$ is $(t-2)$-decomposable and hence shellable. We
also provide a construction demonstrating that this property does not hold for
arbitrary chordal graphs, by exhibiting examples of chordal graphs whose
$t$-clique-free complexes fail to be shellable.

The following lemma is essentially known in the literature. For completeness, we provide a proof here.
\begin{lemma}\label{decomposition}
Let $\Delta$ be a simplicial complex with a sequence of faces $\sigma_1, \ldots, \sigma_m$. Define $\Delta_0 := \Delta$, and for each $1 \leq i \leq m$, set
\[
\Delta_i := \Delta_{i-1} \setminus \sigma_i, \quad 
\Omega_i := \operatorname{link}_{\Delta_{i-1}}(\sigma_i).
\]
Suppose that for each $i = 1, \ldots, m$, the following conditions hold:
\begin{enumerate}
    \item[(i)] $\sigma_i$ is a shedding face of $\Delta_{i-1}$ with $\dim \sigma_i \leq t$;
    \item[(ii)] $\Omega_i$ is $t$-decomposable; and
    \item[(iii)] $\Delta_m$ is $t$-decomposable.
\end{enumerate}
Then $\Delta$ is $t$-decomposable.
\end{lemma}

\begin{proof}
We proceed by reverse induction on $i$. By assumption $(iii)$, $\Delta_m$ is $t$-decomposable. 
Suppose $\Delta_i$ is $t$-decomposable for some $1 \leq i \leq m$. Since $\sigma_i$ is a shedding face of $\Delta_{i-1}$ with $\dim\sigma_i \leq t$ by $(i)$, and $\Omega_i$ is $t$-decomposable by $(ii)$, it follows from the definition of $t$-decomposability that $\Delta_{i-1}$ is also $t$-decomposable. By induction, this holds for all $i$, and in particular, $\Delta_0 = \Delta$ is $t$-decomposable.
\end{proof}

The following result follows immediately from the definition. For completeness, we include a proof below.

\begin{proposition}\label{join}
Let $G_1$ and $G_2$ be graphs with disjoint vertex sets, i.e., $V(G_1) \cap V(G_2) = \emptyset$. Then, for any $t \geq 2$,
$\cf_t(G_1) * \cf_t(G_2) = \cf_t(G_1 \cup G_2).$
\end{proposition}

\begin{proof}
Let $f \in \cf_t(G_1) * \cf_t(G_2)$. Then $f = f_1 \cup f_2$ with $f_1 \in \cf_t(G_1)$ and $f_2 \in \cf_t(G_2)$. Since $V(G_1) \cap V(G_2) = \emptyset$, and both $f_1$ and $f_2$ are $t$-clique-free in their respective graphs, $f$ is $t$-clique-free in $G_1 \cup G_2$. Thus, $f \in \cf_t(G_1 \cup G_2)$.
Conversely, let $f \in \cf_t(G_1 \cup G_2)$. Write $f = f_1 \cup f_2$ with $f_1 \subseteq V(G_1)$ and $f_2 \subseteq V(G_2)$. Since $f$ is $t$-clique-free in $G_1 \cup G_2$, it follows that $f_1 \in \cf_t(G_1)$ and $f_2 \in \cf_t(G_2)$. Hence, $f \in \cf_t(G_1) * \cf_t(G_2)$.
Therefore, $\cf_t(G_1) * \cf_t(G_2) = \cf_t(G_1 \cup G_2)$.
\end{proof}

A graph $G$ is a \emph{$t$-diamond} if it is the union of two distinct $t$-cliques that intersect in a $(t-1)$-clique. Equivalently, $G$ consists of two complete subgraphs on $t$ vertices sharing exactly $t-1$ common vertices.

We are now ready to prove our first main result

\begin{theorem}\label{main}
Let $G$ be a $t$-diamond-free chordal graph. Then, for all $t \geq 3$, the complex $\cf_t(G)$ is $(t\!-\!2)$-decomposable. Consequently, $\cf_t(G)$ is shellable for all $t \geq 3$.
\end{theorem}

\begin{proof}
We prove the theorem by induction on the number of vertices of $G$.
If $|V(G)| = 3$, then $G$ is trivially chordal and contains no $t$-cliques for $t > 3$. Therefore, $\cf_t(G)$ is $(t-2)$-decomposable by direct inspection.
Now assume $|V(G)| \geq 4$, and that the statement holds for all $t$-diamond-free chordal graphs with fewer vertices. Since $G$ is chordal, it contains a simplicial vertex, say $x_1$. Let $N_G[x_1] = \{x_1, x_2, \ldots, x_k\}$.

\vskip 1mm
\noindent
\textsc{Case 1:} $k < t$.

Let $\{x_1, x_2, \ldots, x_p\} \subseteq N_G[x_1]$ be the simplicial vertices of $G$. Consider the induced subgraphs
\[
G_1 = G[\{x_1, \ldots, x_p\}] \quad \text{and} \quad G_2 = G[V(G) \setminus \{x_1, \ldots, x_p\}].
\]
Since $V(G_1) \cap V(G_2) = \emptyset$, it follows from Proposition~\ref{join} that
$\cf_t(G_1 \cup G_2) = \cf_t(G_1) * \cf_t(G_2).$

We claim that $\cf_t(G) = \cf_t(G_1 \cup G_2)$. Suppose not. Then there exists a face $f \in \cf_t(G_1 \cup G_2)$ such that $f \notin \cf_t(G)$. This implies that $f$ contains a $t$-clique in $G$ which is not present in either $G_1$ or $G_2$ individually. However, since $|V(G_1)| = p < k < t$, $G_1$ cannot contain a $t$-clique. Hence, any $t$-clique in $f$ must lie entirely within $G_2$, which contradicts the assumption that $f \notin \cf_t(G)$. Therefore, the claim follows.
Both $G_1$ and $G_2$ are induced subgraphs of $G$, and hence are also $t$-diamond-free and chordal. By the induction hypothesis, both $\cf_t(G_1)$ and $\cf_t(G_2)$ are $(t{-}2)$-decomposable. Then by \cite[Proposition 3.8]{Russ11}, their join $\cf_t(G) = \cf_t(G_1) * \cf_t(G_2)$ is also $(t{-}2)$-decomposable.

\vskip 1mm
\noindent
\textsc{Case 2:} $k \geq t$.  
Define  
\[
S = \left\{ \sigma \subseteq \{x_2, \ldots, x_k\} \,\middle|\, |\sigma| = t - 1 \right\} = \{\sigma_1, \ldots, \sigma_m\}, 
\]  
where $m = \binom{k - 1}{t - 1}$. Note that each $\sigma_i \in \cf_t(G)$ for $1 \leq i \leq m$.

Let $\Delta_0 := \cf_t(G)$. For each $1 \leq i \leq m$, define
\[
\Delta_i := \Delta_{i-1} \setminus \sigma_i, ~
\Omega_i := \operatorname{link}_{\Delta_{i-1}}(\sigma_i).
\]
We claim that each $\sigma_i$ is a shedding face of $\Delta_{i-1}$. Let $\tau \in \operatorname{star}(\Delta_{i-1}, \sigma_i)$, i.e., $\tau = \sigma_i \cup K$ for some $K \subseteq V(G)$ such that no vertex of $K$ lies in $N_G[x_1]$. In particular, $x_1 \notin \tau$.

For any $x \in \sigma_i$, consider the set $(\tau \cup \{x_1\}) \setminus \{x\}$. Since $\tau$ is $t$-clique-free and $x_1$ is not adjacent to any vertex in $K$, this set remains $t$-clique-free. Therefore, $(\tau \cup \{x_1\}) \setminus \{x\} \in \Delta_0$. Since $\Delta_{i-1}$ contains all faces of $\Delta_0$ containing $x_1$, this implies $(\tau \cup \{x_1\}) \setminus \{x\} \in \Delta_{i-1}$. Hence, $\sigma_i$ is a shedding face of $\Delta_{i-1}$. Moreover, $\dim \sigma_i = t - 2$ for all $1 \leq i \leq m$.

Next, we show that
$\operatorname{link}_{\Delta_{i-1}}(\sigma_i) = \cf_t(G \setminus N_G[x_1])$ for all  $1 \leq i \leq m.$
Let $\tau \in \cf_t(G \setminus N_G[x_1])$. Then $\tau \cap \sigma_i = \emptyset$ and $\tau \cup \sigma_i$ is $t$-clique-free in $G$ because $|\sigma_i| = t - 1$ and $G$ is $t$-diamond-free, so all $t$-cliques containing $\sigma_i$ must be in the induced subgraph $G[N_G[x_1]]$. Thus, $\tau \cup \sigma_i \in \Delta_0$, and since $\tau$ is disjoint from $N_G[x_1]$ and $\sigma_i \subseteq N_G[x_1]$, it follows that $\tau \cup \sigma_i \in \Delta_{i-1}$. Hence, $\tau \in \operatorname{link}_{\Delta_{i-1}}(\sigma_i)$.

Conversely, if $\tau \in \operatorname{link}_{\Delta_{i-1}}(\sigma_i)$, then $\tau \cup \sigma_i \in \Delta_{i-1} \subseteq \Delta_0$ and is $t$-clique-free. Since $\tau \cap \sigma_i = \emptyset$ and $\sigma_i \subseteq N_G[x_1]$, it follows that $\tau$ is disjoint from $N_G[x_1]$, and hence $\tau \in \cf_t(G \setminus N_G[x_1])$. This completes the claim.

Now $G \setminus N_G[x_1]$ must be $t$-diamond-free and chordal. Then by induction, $\cf_t(G \setminus N_G[x_1])$ is $(t{-}2)$-decomposable. Since each $\Omega_i$ is isomorphic to this complex, they are all $(t{-}2)$-decomposable.

Now consider the set  
\[
S' = \left\{ X \subseteq \{x_1, \ldots, x_k\} \;\middle|\; X = \{x_1\} \cup X',\; X' \subseteq \{x_2,\ldots,x_k\},\; |X'| = t - 2 \right\}.
\]  
This set can be enumerated as \( S' = \{X_1, \ldots, X_\ell\} \), where \( \ell = \binom{k - 1}{t - 2} \). Let $\Delta'$ be the simplicial complex with facets $X_1, \ldots, X_\ell$. Note that $\Delta' = \{x_1\} * \Delta''$, where $\Delta''$ is the $(t-3)$-skeleton of the simplex on $\{x_2, \ldots, x_k\}$. Hence, by \cite[Proposition 3.8 and Lemma 3.10]{Russ11}, $\Delta'$ is $(t{-}2)$-decomposable.

We now claim that  
$\Delta_m = \Delta' * \cf_t(G \setminus N_G[x_1]).$
Let $f = f_1 \cup f_2 \in \Delta' * \cf_t(G \setminus N_G[x_1])$. Then $\sigma_i \nsubseteq f_1$, so $\sigma_i \nsubseteq f$ for all $1 \leq i \leq m$. Thus, $f \in \Delta_m$.

Conversely, let $f \in \Delta_m$. Define $f_1 := f \cap N_G[x_1]$ and $f_2 := f \setminus f_1$. Since $f$ is $t$-clique-free, we must have $|f_1| \leq t - 1$. Moreover, since $\sigma_i \nsubseteq f$, $f_1$ cannot contain any $(t{-}1)$-subset of $\{x_2, \ldots, x_k\}$, and hence $f_1 \in \Delta'$. Clearly, $f_2 \in \cf_t(G \setminus N_G[x_1])$. Therefore, $f \in \Delta' * \cf_t(G \setminus N_G[x_1])$.

Hence, $\Delta_m$ is the join of two $(t{-}2)$-decomposable complexes, and thus $(t{-}2)$-decomposable. By Lemma~\ref{decomposition}, it follows that $\cf_t(G) = \Delta_0$ is also $(t{-}2)$-decomposable.
\end{proof}

A graph $G$ is a \emph{block graph} if every \emph{block} (a maximal connected subgraph that remains connected upon the removal of any single vertex) is a complete graph. Since block graphs are $t$-diamond-free and chordal for all $t \geq 3$, we obtain the following immediate consequence of Theorem~\ref{main}:
\begin{corollary}\label{cor:block-graph}
For any block graph $G$, the complex $\cf_t(G)$ is $(t-2)$-decomposable and hence shellable.
\end{corollary}

The following corollary is an immediate consequence of Lemma 3.6 in \cite{fv} and Theorem~\ref{main}: Let $I$ be a squarefree monomial ideal. Then the quotient ring $\K[x_1,\ldots,x_n]/I$ is Cohen-Macaulay if and only if it is sequentially Cohen-Macaulay and the ideal $I$ is unmixed.

\begin{corollary}
Let $G$ be a $t$-diamond-free chordal graph. Then the edge ideal $I(\mathcal{CH}_t(G))$ is Cohen-Macaulay if and only if $I(\mathcal{CH}_t(G))$ is unmixed.
\end{corollary}

The following construction demonstrates that there exist chordal graphs whose $t$-clique-free complexes are not sequentially Cohen-Macaulay (and hence not shellable) for certain parameters.

\begin{example}\label{not-shellable}
Let $K_n$ be the complete graph on $n$ vertices with vertex set  
$V(K_n) = \{y_1, y_2, \ldots, y_n\}.$
We construct a graph $G_{n,t}$ by augmenting $K_n$ as follows:  
For each $1 \leq i \leq n$, introduce $t - 2$ new vertices  
$\{x_{i,1}, x_{i,2}, \ldots, x_{i,t-2}\}, $
so that the vertex set of $G_{n,t}$ becomes  
\[ V(G_{n,t}) = V(K_n) \cup \bigcup_{i=1}^n \{x_{i,1}, x_{i,2}, \ldots, x_{i,t-2}\}. \]  

The edge set of $G_{n,t}$ consists of:  
\begin{enumerate}
    \item All edges of $K_n$, and  
    \item For each $i$, a $t$-clique formed by $\{y_i, y_{i+1}, x_{i,1}, \ldots, x_{i,t-2}\}$, where indices are taken modulo $n$.
\end{enumerate}  
Thus,  
$E(G_{n,t}) = E(K_n) \cup \bigcup_{i=1}^n \left\{ \{x, y\} \mid x, y \in \{y_i, y_{i+1}, x_{i,1}, \ldots, x_{i,t-2}\}, x \neq y \right\}.$

Assume $n \neq 3,5$ and $t > \lceil \frac{n}{2} \rceil$. We show that the $t$-clique-free complex $\cf_t(G_{n,t})$ is not sequentially Cohen-Macaulay.  
Suppose, for contradiction, that $\cf_t(G_{n,t})$ is sequentially Cohen-Macaulay. By \cite[Theorem 3.30]{J08}, the link of any face in $\cf_t(G_{n,t})$ must also be sequentially Cohen-Macaulay.  

Consider the face  
$\sigma = \{x_{i,j} \mid 1 \leq i \leq n,\ 1 \leq j \leq t-2\},$
which consists of all the newly added vertices. A key observation is that  
$\link_{\cf_t(G_{n,t})}(\sigma) = \ind(C_n),$
where $\ind(C_n)$ is the independence complex of the $n$-cycle.  
However, by \cite[Theorem 10]{russ}, $\ind(C_n)$ is \textit{not} sequentially Cohen-Macaulay for $n \neq 3,5$, yielding a contradiction.  
Thus, $\cf_t(G_{n,t})$ fails to be sequentially Cohen-Macaulay for $t > \lceil \frac{n}{2} \rceil$.  
\end{example}

\section{Shellability Preservation under Clique Attachments and Whiskerings}
\label{constructions}

In this section, we study the preservation of shellability of clique-free
complexes under various graph operations. We begin by introducing the clique
attachment construction $\cl(H,S,t)$ (Construction~\ref{cons:graph-attachment})
and establish a key structural result: the clique-free complex $\cf_t(G)$ is
shellable if and only if $\cf_t(H \setminus S)$ is shellable
(Theorem~\ref{she-extension}). This result generalizes and unifies earlier work
on shellability (see Corollary~\ref{cor:FH-VanVilla}) and yields new sufficient
conditions for shellability (Corollary~\ref{cor:clique-ext}). We then investigate
clique whiskerings $G(\Pi,t)$, showing that their associated clique-free
complexes are pure (Theorem~\ref{pure}), $(t-2)$-decomposable
(Theorem~\ref{clique-vertex}), and Cohen--Macaulay
(Corollary~\ref{CM}).

\begin{cons}\label{cons:graph-attachment}
Let $H$ be a graph and $S \subseteq V(H)$. We define the graph $G = \mathrm{Cl}(H, S, t)$ as follows:
\begin{itemize}
    \item For each vertex $v \in S$, attach a new clique $K_v$ of size at least $t$, where $v \in V(K_v)$ (i.e., $K_v$ contains $v$ and $|V(K_v)| \geq t$).
    \item $K_v$ and $K_u$ does not share any vertex for $u \neq v$.
    \item The sizes of the attached cliques $\{K_v\}_{v \in S}$ may vary (i.e., $|V(K_v)|$ is independent of $|V(K_u)|$ for $u \neq v$).
\end{itemize}
All attachments are performed simultaneously, meaning the resulting graph $G$ is obtained by adding all cliques $\{K_v\}_{v \in S}$ to $H$ at once.
\end{cons}

We now establish one of the main structural results of this section: a complete characterization of shellability preservation under clique attachments.

\begin{theorem}\label{she-extension}
Let $G = \cl(H, S, t)$, where $H$ is a graph, $S \subseteq V(H)$, and $t \geq 2$. Then  $\cf_t(H \setminus S)$ is shellable if and only if $\cf_t(G)$ is shellable.
\end{theorem}

\begin{proof}
Assume  $\cf_t(H \setminus S)$ is shellable.
We proceed by induction on $|S|$.
If $|S| = 0$, then $G = H = H \setminus S$, so the conclusion holds trivially. 
Now suppose the statement holds for all sets $S$ of size $n$. Let $|S| = n+1$ and pick a vertex $v \in S$. Let $V(K_v) = \{x_1, x_2, \ldots, x_k\}$ with $x_k = v$ and $k \geq t$.
Define the collection of all $(t-1)$-subsets of $\{x_2, \ldots, x_k\}$:
\[
\mathcal{S} = \left\{ \sigma \subseteq \{x_2, \ldots, x_k\} \mid |\sigma| = t - 1 \right\} = \{\sigma_1, \ldots, \sigma_m\},
\]
where $m = \binom{k - 1}{t - 1}$. Each $\sigma_i$ belongs to $\cf_t(G)$ since they avoid forming a $t$-clique.

Let $\Delta_0 := \cf_t(G)$. For $1 \leq i \leq m$, define:
\[
\Delta_i := \Delta_{i-1} \setminus \sigma_i, ~
\Omega_i := \operatorname{link}_{\Delta_{i-1}}(\sigma_i).
\]

We claim that each $\sigma_i$ is a shedding face of $\Delta_{i-1}$. Indeed, for any $\tau \in \operatorname{star}(\Delta_{i-1}, \sigma_i)$, write $\tau = \sigma_i \cup K$ where $K$ contains no vertex of $N_G[x_1]$. Then for any $x \in \sigma_i$, the set $(\tau \cup \{x_1\}) \setminus \{x\}$ remains $t$-clique-free: $x_1$ is not adjacent to any vertex in $K$, and $\tau$ is $t$-clique-free by assumption. Hence, $(\tau \cup \{x_1\}) \setminus \{x\} \in \Delta_{0}$, and thus belongs to $\Delta_{i-1}$, proving that $\sigma_i$ is a shedding face of $\Delta_{i-1}$.

Next, we show:
$\operatorname{link}_{\Delta_{i-1}}(\sigma_i) = \cf_t(G \setminus N_G[x_1]).$
First, if $\tau \in \cf_t(G \setminus N_G[x_1])$, then $\tau \cap \sigma_i = \emptyset$ and all $t$-cliques containing $\sigma_i$ lie within $G[N_G[x_1]]$. Thus, $\tau \cup \sigma_i$ is $t$-clique-free in $G$, so $\tau \in \operatorname{link}_{\Delta_{i-1}}(\sigma_i)$. Conversely, any $\tau$ in the link satisfies that $\tau \cup \sigma_i$ is $t$-clique-free, and since $\sigma_i \subseteq N_G[x_1]$, it follows that $\tau$ is disjoint from $N_G[x_1]$, hence $\tau \in \cf_t(G \setminus N_G[x_1])$.
Observe that:
$G \setminus N_G[x_1] = \cl(H \setminus v, S \setminus \{v\}, t).$
Since $|S \setminus \{v\}| = n$ and
$(H \setminus v) \setminus (S \setminus \{v\}) = H \setminus S,$
the induction hypothesis gives that $\cf_t(G \setminus N_G[x_1])$ is shellable.

Now define:
\[
\mathcal{S}' = \left\{ X \subseteq \{x_1, \ldots, x_k\} \mid X = \{x_1\} \cup X',\; X' \subseteq \{x_2,\ldots,x_k\},\; |X'| = t - 2 \right\}.
\]
Let $\Delta'$ be the simplicial complex whose facets are the elements of $\mathcal{S}'$. Then $\Delta' = \{x_1\} * \Delta''$, where $\Delta''$ is the $(t-3)$-skeleton of the simplex on $\{x_2, \ldots, x_k\}$. By \cite[Proposition 3.8 and Lemma 3.10]{Russ11}, $\Delta'$ is shellable.

We now claim:
$\Delta_m = \Delta' * \cf_t(G \setminus N_G[x_1]).$
Let $f = f_1 \cup f_2 \in \Delta' * \cf_t(G \setminus N_G[x_1])$. Then for all $i$, $\sigma_i \nsubseteq f_1$, so $\sigma_i \nsubseteq f$, hence $f \in \Delta_m$. Conversely, if $f \in \Delta_m$, let $f_1 = f \cap N_G[x_1]$ and $f_2 = f \setminus f_1$. Since $f$ is $t$-clique-free and none of the $\sigma_i$ are subsets of $f_1$, we have $f_1 \in \Delta'$. Also, $f_2 \in \cf_t(G \setminus N_G[x_1])$. Thus, $f \in \Delta' * \cf_t(G \setminus N_G[x_1])$.

Finally, since both $\Delta'$ and $\cf_t(G \setminus N_G[x_1])$ are shellable by the induction hypothesis, their join $\Delta_m$ is shellable by properties of shellable complexes. Through iterative application of \cite[Lemma 3.4]{Russ11}, we obtain in sequence that $\Delta_{m-1}$, $\Delta_{m-2}$, and continuing to $\Delta_0 = \cf_t(G)$ are all shellable.

Suppose $\cf_t(G)$ is shellable.  Let $S = \{v_1, v_2, \ldots, v_n\} \subseteq V(H)$. For each $i$, choose a subset $\sigma_i \subset V(K_{v_i} \setminus \{v_i\})$ with $|\sigma_i| = t - 1$, and define $\sigma = \bigcup_{i=1}^n \sigma_i$. By the construction of $G = \cl(H, S, t)$, it follows that $\sigma \in \cf_t(G)$.

    We claim that 
    $\operatorname{link}_{\cf_t(G)}(\sigma) = \cf_t(H \setminus S).$ 
    First, let $f \in \operatorname{link}_{\cf_t(G)}(\sigma)$. Then $f \cup \sigma \in \cf_t(G)$ and $f \cap \sigma = \emptyset$. Since $\sigma$ is supported entirely in the added cliques and $f$ is disjoint from $\sigma$, it must be that $f \in \cf_t(H \setminus S)$.

    Conversely, suppose $g \in \cf_t(H \setminus S)$. Then:
    \begin{enumerate}
        \item $g \cap \sigma = \emptyset$, since the support of $\cf_t(H \setminus S)$ lies entirely outside the added clique vertices used to form $\sigma$;
        \item $g \cup \sigma \in \cf_t(G)$, because both $g$ and $\sigma$ belong to $\cf_t(G)$ and are disjoint.
    \end{enumerate}
    Thus, $g \in \operatorname{link}_{\cf_t(G)}(\sigma)$.

    Therefore, $\operatorname{link}_{\cf_t(G)}(\sigma) = \cf_t(H \setminus S)$, which is shellable by \cite[Theorem 3.30]{J08}. It follows that $\cf_t(H\setminus S)$ is shellable.
\end{proof}
As a direct consequence of Theorem~\ref{she-extension}, we recover the results of Francisco-H\'a~\cite{FH} and Van Tuyl-Villarreal~\cite{VanVilla}:

\begin{corollary}\label{cor:FH-VanVilla}
\mbox{}
\begin{enumerate}[label=(\alph*)]
    \item \cite[Theorem 3.3]{FH} Let $G$ be a graph and let $S\subseteq V(G)$. Suppose $G\setminus S$ is a chordal graph
or a five-cycle $C_5$. Then $\cf_2(G,S,2)$ is a sequentially Cohen-Macaulay graph.
    
    \item \cite[Corollary 2.7]{VanVilla} Let $G$ be a graph and let $S\subseteq V(G)$. If $\cf_2(\cl(G,S,2))$ is shellable, then $\cf_2(G\setminus S)$ is
shellable.
\end{enumerate}
\end{corollary}

Let $G = (V(G), E(G))$ be a graph. A \emph{cycle cover} of $G$ is a subset $S \subseteq V(G)$ such that every cycle in $G$ contains at least one vertex from $S$. Equivalently, $S$ is a cycle cover if and only if the induced subgraph $G \setminus S$ is a forest.
If $S$ is a \emph{vertex cover} of $G$ (i.e., every edge of $G$ has at least one endpoint in $S$), then $G \setminus S$ is an independent set. In particular, since an independent set is a forest, every vertex cover is also a cycle cover. The converse, however, does not hold in general.

As an immediate consequence of Theorem~\ref{she-extension}, we obtain several fundamental results about shellability preservation under clique addition in clique-free complexes. These provide constructive criteria for generating new shellable complexes.

\begin{corollary}\label{cor:clique-ext}
    Let $G = \cl(H, S, t)$, where $H$ is a graph, $S \subseteq V(H)$, and $t \geq 2$. Then $\cf_t(G)$ is shellable in any of the following cases:
    \begin{enumerate}
        \item $S$ is a cycle cover of $H$;
        \item $H \setminus S$ is a $t$-diamond-free chordal graph (in particular, a block graph);
        \item $|S| \geq |V(H)| - 3$.
    \end{enumerate}
\end{corollary}

\begin{proof}
    (1) If $S$ is a cycle cover of $H$, then $H \setminus S$ is a forest. When $t = 2$, it follows from \cite[Corollary 7]{Wood2} that $\cf_t(H \setminus S)$ is vertex decomposable. For $t \geq 3$, $\cf_t(H \setminus S)$ is a simplex. In both cases, by Theorem~\ref{she-extension}, $\cf_t(G)$ is shellable.

    \medskip
    \noindent
    (2) If $H \setminus S$ is a $t$-diamond-free chordal graph, then by Theorem~\ref{main}, $\cf_t(H \setminus S)$ is shellable. Applying Theorem~\ref{she-extension}, it follows that $\cf_t(G)$ is also shellable.

    \medskip
    \noindent
    (3) If $|S| \geq |V(H)| - 3$, then $H \setminus S$ has at most three vertices. Hence, it is either a triangle, a forest, or a graph with isolated vertices—all of which are $t$-diamond-free and chordal. By Theorems~\ref{main} and~\ref{she-extension}, it follows that $\cf_t(G)$ is shellable.
\end{proof}
As an immediate consequence of Theorem~\ref{she-extension}, we can identify certain vertex sets such that adding cliques on these sets does not yield a shellable clique-free complex.

\begin{corollary}
Let $G$ be a graph, $S \subseteq V(G)$, and $t \geq 3$.  
If $G \setminus S = G_{n,t}$, where $G_{n,t}$ is the graph constructed in Example~\ref{not-shellable}, $n \neq 3,5$, and $t > \left\lceil \frac{n}{2} \right\rceil$, then $\cf_t(G, S, t)$ is not shellable.
\end{corollary}

In \cite{CookNagel}, Cook II and Nagel introduced the notion of a \emph{clique whiskering} of a graph. We present the following generalization of their construction.

\begin{cons}\label{CN}
    Let $G$ be a graph. A \emph{clique vertex-partition} of $G$ is a collection $\Pi = \{W_1, \ldots, W_p\}$ of pairwise disjoint (possibly empty) cliques in $G$ whose union is $V(G)$. Every graph admits at least one such partition - for instance, the \emph{trivial partition}
\[
\Pi = \{\{x_1\}, \ldots, \{x_n\}\}, \quad \text{where } V(G) = \{x_1, \ldots, x_n\}.
\]

Given a graph $G$, a clique vertex-partition $\Pi = \{W_1, \ldots, W_p\}$, and an integer $t \geq 2$, the \emph{$t$-clique whiskering} of $G$ with respect to $\Pi$, denoted $G(\Pi, t)$, is constructed as follows:

\begin{itemize}
    \item Vertex set: For each $1 \leq i \leq p$, introduce $t-1$ new vertices $\{x_{i,1}, \ldots, x_{i,t-1}\}$ all distinct and disjoint from $V(G)$. Then,
    $V(G(\Pi, t)) = V(G) \cup \bigcup_{i=1}^p \{x_{i,1}, \ldots, x_{i,t-1}\}.$
  
    \item Edge set: The edge set extends $E(G)$ by turning each $W_i \cup \{x_{i,1}, \ldots, x_{i,t-1}\}$ into a clique:
    \[
    E(G(\Pi, t)) = E(G) \cup \bigcup_{i=1}^p \left\{ \{a, b\} \mid a \neq b, a, b \in W_i \cup \{x_{i,1}, \ldots, x_{i,t-1}\} \right\}.
    \]
\end{itemize}

\end{cons}

\begin{remark}
When $t = 2$, this construction specializes to the classical clique whiskering introduced by Cook~II and Nagel \cite{CookNagel}. The case $t > 2$ provides a natural higher-order generalization.
\end{remark}

We extend \cite[Lemma 3.2]{CookNagel} to higher clique whiskerings, showing that for any $t \geq 2$, the complex $\cf_t(G(\Pi, t))$ has dimension $p(t-1)-1$ and is pure.
\begin{theorem}\label{pure}
    Let $G$ be a graph with a clique vertex-partition $\Pi = \{W_1, \ldots, W_p\}$. 
    For every integer $t \geq 2$, the $t$-clique whiskering $G(\Pi, t)$ satisfies:
    \begin{enumerate}
        \item $\dim \big( \cf_t(G(\Pi, t)) \big) = p(t-1) - 1$;
        \item The complex $\cf_t(G(\Pi, t))$ is pure.
    \end{enumerate}
\end{theorem}

\begin{proof}
    (1)  For each $W_i$, let $\{x_{i,1}, \ldots, x_{i,t-1}\}$ be the newly added vertices, and let $B_{W_i}$ denote the clique in $G(\Pi, t)$ formed by $W_i \cup \{x_{i,1}, \ldots, x_{i,t-1}\}$.
Let $F \in \cf_t(G(\Pi, t))$ be any face. Since $F$ is $t$-clique-free, it contains at most $t-1$ vertices from each $B_{W_i}$ (as any $t$ vertices in $B_{W_i}$ form a $t$-clique). Thus, $|F| \leq p(t-1)$, meaning $\dim(\cf_t(G(\Pi, t))) < p(t-1)$.
However, the set $\bigcup_{i=1}^p \{x_{i,1}, \ldots, x_{i,t-1}\}$ is $t$-clique-free and has $p(t-1)$ vertices. Therefore,
    $\dim \big( \cf_t(G(\Pi, t)) \big) = p(t-1) - 1.$

    \vskip 1mm 
    \noindent
    (2)  We show that every maximal $t$-clique-free set in $G(\Pi, t)$ has exactly $p(t-1)$ vertices.
Let $I$ be any $t$-clique-free set in $G(\Pi, t)$. We can decompose $I$ as $I = \bigcup_{i=1}^p I_i$, where $I_i \subseteq B_{W_i}$ (possibly empty) and $|I_i| \leq t-1$ for all $i$.
For each $i$, augment $I_i$ with a subset $S_i \subseteq \{x_{i,1}, \ldots, x_{i,t-1}\} \setminus I_i$ such that $|I_i \cup S_i| = t-1$ (if $|I_i| = t-1$, set $S_i = \emptyset$). Since the $\{x_{i,j}\}$ are disjoint across different $W_i$, the union $\bigcup_{i=1}^p (I_i \cup S_i)$ remains $t$-clique-free and has size $p(t-1)$.
Moreover, any vertex $x$ not in $\bigcup_{i=1}^p (I_i \cup S_i)$ must belong to some $B_{W_i}$. Adding $x$ to $I_i \cup S_i$ would create a $t$-clique in $B_{W_i}$, so $\bigcup_{i=1}^p (I_i \cup S_i)$ is maximal.

    Since $I$ was arbitrary, all maximal $t$-clique-free sets have size $p(t-1)$, proving that $\cf_t(G(\Pi, t))$ is pure.
\end{proof}

Recall from \cite[Definition 4.2]{Russ11} that a vertex $v \in V(\mathcal{C})$ in a clutter $\mathcal{C}$ is called a \emph{simplicial vertex} if for any two edges $e_1, e_2 \in E(\mathcal{C})$ containing $v$, there exists an edge $e_3 \in E(\mathcal{C})$ such that 
$ e_3 \subseteq (e_1 \cup e_2) \setminus \{v\}.$

The following result generalizes \cite[Theorem 3.3]{CookNagel}, where it was shown that 2-clique-free complexes are vertex-decomposable. We extend this to higher cliques, proving that for $t>2$, $t$-clique-free complexes are shellable.

\begin{theorem}\label{clique-vertex}
Let $G$ be a graph and let $\Pi$ be a clique vertex partition of $G$. For each $t \geq 2$, 
the complex $\cf_t(G(\Pi, t))$ is $(t{-}2)$-decomposable. Consequently, $\cf_t(G(\Pi, t))$ is shellable.
\end{theorem}

\begin{proof}
Let $\CH_t(G(\Pi, t))$ denote the $t$-clique clutter of $G(\Pi, t)$, where $\Pi = \{W_1, \ldots, W_p\}$. By Proposition~\ref{ind} and \cite[Corollary~5.3]{Russ11}, it suffices to show that every contraction of $\CH_t(G(\Pi, t))$ has a simplicial vertex.
Let $\mathcal{H}$ be a contraction of $\CH_t(G(\Pi, t))$, say $\mathcal{H} = \CH_t(G(\Pi, t)) / V_c$ for some subset $V_c \subseteq V(G(\Pi, t))$. The vertex set of $G(\Pi, t)$ is
\[
V(G(\Pi, t)) = V(G) \cup \bigcup_{i=1}^p \{x_{i,1}, x_{i,2}, \ldots, x_{i,t-1}\},
\]
and define
$N = \bigcup_{i=1}^p \{x_{i,1}, x_{i,2}, \ldots, x_{i,t-1}\}.$
We consider two cases:

\smallskip

\noindent \textbf{Case 1:} $N \cap V(\mathcal{H}) = \emptyset$, i.e., all vertices in $N$ are contracted.

Let $a \in V(\mathcal{H})$. Then $a \in W_i$ for some $1 \leq i \leq p$. Since the set $\{a, x_{i,1}, \ldots, x_{i,t-1}\}$ forms a $t$-clique in $G(\Pi, t)$, and hence is an edge in $\CH_t(G(\Pi, t))$. As all $x_{i,j}$ are contracted, the contraction $\mathcal{H}$ contains the singleton edge $\{a\}$. Thus, all edges in $\mathcal{H}$ are singletons, so every vertex is simplicial.

\smallskip

\noindent \textbf{Case 2:} $N \cap V(\mathcal{H}) \neq \emptyset$.

Suppose first that $\{a\} \in E(\mathcal{H})$ for every $a \in N \cap V(\mathcal{H})$. We claim that every edge in $\mathcal{H}$ is a singleton.
Let $\alpha \in V(\mathcal{H})$. If $\alpha \in N$, then by assumption, $\{\alpha\} \in E(\mathcal{H})$. Suppose $\alpha \notin N$, so $\alpha \in W_i$ for some $i$. If $\{x_{i,1}, \ldots, x_{i,t-1}\} \cap V(\mathcal{H}) = \emptyset$, then as in Case 1, we get $\{\alpha\} \in E(\mathcal{H})$.

Otherwise, there exists $\beta \in \{x_{i,1}, \ldots, x_{i,t-1}\} \cap V(\mathcal{H})$. Since $\{\beta\} \in E(\mathcal{H})$, there exists $f \in E(\CH_t(G(\Pi, t)))$ with $\beta \in f \subseteq W_i \cup \{x_{i,1}, \ldots, x_{i,t-1}\}$ and $f \cap V(\mathcal{H}) = \{\beta\}$. Consider $f' = (f \cup \{\alpha\}) \setminus \{\beta\}$. Then $f'$ is also a $t$-clique in $G(\Pi, t)$, and so $f' \in E(\CH_t(G(\Pi, t)))$. The image $e' = f' \cap V(\mathcal{H}) \subseteq \{\alpha\}$ implies $\{\alpha\} \in E(\mathcal{H})$.

Hence, every edge of $\mathcal{H}$ is a singleton and every vertex is simplicial.

\smallskip

\noindent Now suppose there exists $\alpha \in N \cap V(\mathcal{H})$ such that $\{\alpha\} \notin E(\mathcal{H})$. We claim that $\alpha$ is simplicial in $\mathcal{H}$.
Assume $\alpha$ belongs to two distinct edges $e_1, e_2 \in E(\mathcal{H})$ such that $\alpha \in e_1 \cap e_2$ and $e_1 \neq e_2$. Then there exist $f_1, f_2 \in E(\CH_t(G(\Pi, t)))$ with $f_j \subseteq W_i \cup \{x_{i,1}, \ldots, x_{i,t-1}\}$ and $f_j \cap V(\mathcal{H}) = e_j$ for $j = 1, 2$.
Since $f_1 \neq f_2$, we have $|f_1 \cup f_2| \geq t{+}1$. Choose a $t$-subset $f_3 \subseteq (f_1 \cup f_2) \setminus \{\alpha\}$ such that $f_3 \in E(\CH_t(G(\Pi, t)))$, and let $e_3 = f_3 \cap V(\mathcal{H})$.
Then $e_3 \subseteq (e_1 \cup e_2) \setminus \{\alpha\}$. If $e_3 \in E(\mathcal{H})$, then $\alpha$ is simplicial. If not, since $\mathcal{H}$ is a contraction, there exists $e_3' \in E(\mathcal{H})$ with $e_3' \subseteq e_3$, again showing $\alpha$ is simplicial.

Thus, in all cases, $\mathcal{H}$ has a simplicial vertex. This completes the proof.
\end{proof}

Theorems~\ref{pure} and~\ref{clique-vertex} together imply the following corollary. 
When $t = 2$, this specializes to \cite[Corollary 3.5]{CookNagel}.

\begin{corollary}\label{CM}
Let $G$ be a graph with a clique vertex-partition $\Pi$. 
Then for every integer $t \geq 2$, the complex $\cf_t(G(\Pi,t))$ is Cohen-Macaulay.
\end{corollary}
The following establishes the linear resolution property for Alexander duals of clique complexes:

\begin{corollary}
Let $G$ be a graph on $n$ vertices with a clique vertex-partition $\Pi$. For any integer $t \geq 2$, the ideal $I_{\cf_t(G(\Pi, t))^\vee}$ has a linear resolution with
$\reg\left(I_{\cf_t(G(\Pi, t))^\vee}\right) = n.$
\end{corollary}

\begin{proof}
By Corollary~\ref{CM} and \cite[Theorem 3]{eagon}, the ideal $I_{\cf_t(G(\Pi, t))^\vee}$ has a linear resolution. 
For the regularity, Theorem~\ref{pure} gives $\dim(\cf_t(G(\Pi,t))) = p(t-1) - 1$ while $|V(G(\Pi,t))| = n + p(t-1)$. Thus all minimal generators of $I_{\cf_t(G(\Pi, t))^\vee}$ have degree $n$, and consequently $\reg\left(I_{\cf_t(G(\Pi, t))^\vee}\right) = n$.
\end{proof}

\section{Linear Resolutions of Higher-Order Clique Ideals}\label{algebra}

In this section, we prove that for any chordal graph $G$, the edge ideal $I(\overline{\mathcal{CH}_t(G)})$ of the complement $t$-clique clutter admits a $t$-linear resolution over all fields. This extends Fr\"oberg's classic theorem from edge ideals to higher-degree generated ideals associated with $t$-clique structures.

\begin{definition}
Let $\mathcal{C}$ be a $t$-uniform clutter.
\begin{enumerate}
    \item The \emph{closed neighborhood} of a $(t-1)$-subset $e \subset V(\mathcal{C})$ is:
    \[
    N_{\mathcal{C}}[e] = e \cup \{v \in V(\mathcal{C}) \mid e \cup \{v\} \in E(\mathcal{C})\}.
    \]
    
    \item A $(t-1)$-subset $e$ is a \emph{maximal subedge} if $N_{\mathcal{C}}[e] \neq e$ (equivalently, if $e$ is properly contained in some edge of $\mathcal{C}$).
    
    \item A maximal subedge $e$ is \emph{simplicial} if its closed neighborhood $N_{\mathcal{C}}[e]$ forms a clique in $\mathcal{C}$.
    
    \item The \emph{deletion} $\mathcal{C} \setminus e$ of a subset $e \subseteq V(\mathcal{C})$ is the clutter with edge set:
    $\{F \in E(\mathcal{C}) \mid e \not\subseteq F\}.$
    
    \item A clutter $\mathcal{C}$ is \emph{chordal} if either:
    \begin{itemize}
        \item $\mathcal{C}$ has no edges, or
        \item There exists a simplicial maximal subedge $e$ such that $\mathcal{C} \setminus e$ is chordal.
    \end{itemize}
\end{enumerate}
\end{definition}

We now prove the main result of this section.
\begin{theorem}\label{linear-rs}
Let $G$ be a chordal graph. Then for every integer $t\ge2$, the edge ideal
$I\bigl(\overline{\mathcal{CH}_t(G)}\bigr)$
has a $t$-linear resolution over any field.
\end{theorem}

\begin{proof}
By \cite[Theorem~3.3]{BYZ17}, it suffices to show that the clutter $\mathcal{CH}_t(G)$ is chordal for all $t\ge2$. We prove this by induction on the number of vertices of $G$.
If $|V(G)| = 2$, then $G$ has at most one edge. Consequently, $\mathcal{CH}_t(G)$ has at most one edge (and is empty whenever $t>2$), so it is trivially chordal.
Now suppose $|V(G)| \ge 3$, and assume the claim holds for all chordal graphs with fewer vertices. Let $x_1$ be a simplicial vertex of $G$, and set $G' = G\setminus\{x_1\}$.
If $\deg_G(x_1)<t-1$, then no $t$-clique in $G$ contains $x_1$. Hence $\mathcal{CH}_t(G) = \mathcal{CH}_t(G')$, and the result follows from the induction hypothesis.

If $\deg_G(x_1) = d-1 \geq t-1$, let $N_G(x_1) = \{x_2, \dots, x_d\}$. Define the family of $(t-1)$-sets:
\[
\mathcal{X} = \big\{ e \subseteq \{x_1, \dots, x_{d-1}\} \mid |e| = t-1 \text{ and } x_1 \in e \big\}.
\]
Then $|\mathcal{X}| = \binom{d-2}{t-2}$ since we must choose $t-2$ additional vertices from the $d-2$ neighbors of $x_1$ (excluding $x_1$ itself).

Each $e \in \mathcal{X}$ has a unique representation $e = \{x_1, x_{i_2}, \dots, x_{i_{t-1}}\}$ where $2 \leq i_2 < \cdots < i_{t-1} \leq d-1$. The lexicographic order on $\mathcal{X}$ is defined as follows: for $e = \{x_1, x_{i_2}, \dots, x_{i_{t-1}}\}$ and $e' = \{x_1, x_{j_2}, \dots, x_{j_{t-1}}\}$,
\[
e < e' \iff \exists \ell \in \{2, \dots, t-1\} \text{ such that } i_\ell < j_\ell \text{ and } i_k = j_k \text{ for all } k < \ell.
\]

For each $1 \leq i \leq \binom{d-2}{t-2}$, define the clutter
$\Omega_i(G) = \mathcal{CH}_t(G) \setminus e_1 \setminus \cdots \setminus e_{i-1}.$
We make the following claims:
\begin{enumerate}
    \item $e_i$ is a maximal simplicial $(t-1)$-subedge in $\Omega_i(G)$ for all $i$;
    \item $\Omega_m(G) = \mathcal{CH}_t(G\setminus x_1)$ where $m = \binom{d-2}{t-2}$.
\end{enumerate}
These claims together imply that $\mathcal{CH}_t(G)$ is chordal by induction on $|V(G)|$.

Consider the closed neighborhood $N_{\Omega_i(G)}[e_i] = e_i \cup S$ where $S \subseteq \{x_2,\ldots,x_d\}$. 
Since $N_G[x_1]$ forms a clique, $e_i \cup \{x_d\}$ is a edge of $\CH_t(G)$ and none of $e_j$ (where $j<i$) is in $e_i \cup \{x_d\}$, which implies $e_i \cup \{x_d\}$ is also an edge of $\Omega_i(G)$. This implies $x_d \in S$, and consequently $N_{\Omega_i(G)}[e_i] \supsetneq e_i$. Therefore, $e_i$ is a maximal subedge of $\Omega_i(G)$ for all $1 \leq i \leq m$.

We show that $e_i$ is simplicial in $\Omega_i(G)$ for all $1 \leq i \leq m$. The closed neighborhood $N_{\Omega_i(G)}[e_i]$ consists of $e_i$ together with a subset $S \subseteq \{x_2, \ldots, x_d\}$. Let $e_i$ and $e_{i-1}$ share their first $k-1$ entries, with their $k$-th entries being $x_{k_i}$ and $x_{k_{i-1}}$, respectively, where $k_i > k_{i-1}$.

First, suppose for contradiction that some $x_s \in S$ satisfies $s < k_i$. Then $e_i \cup \{x_s\}$ forms a $t$-clique, which must contain a subedge $e_r = (e_i \cup \{x_s\}) \setminus \{x_\ell\}$ for some $\ell$. Since $s < k_i$, the lexicographic order ensures $e_r < e_i$. But $\Omega_i(G)$ excludes all such $e_r$ (as they appear in $\{e_1, \ldots, e_{i-1}\}$), contradicting the assumption that $e_i \cup \{x_s\}$ is present in $\Omega_i(G)$. Thus, $S$ can only contain vertices $x_s$ with $s > k_i$.

Now consider any $t$-clique $C$ containing $x_1$ in $e_i \cup S$. Because $G[\{x_1, \ldots, x_d\}]$ is complete, $C$ is formed by extending $e_i$ with vertices from $S$, all of which have indices larger than $k_i$. By the lexicographic ordering, no such $C$ can contain any earlier $e_j$ (where $j < i$) as a subset. This means $N_{\Omega_i(G)}[e_i]$ induces a clique in $\Omega_i(G)$, and therefore $e_i$ is simplicial.

We now prove that $\Omega_m(G) = \mathcal{CH}_t(G\setminus x_1)$. Consider any $t$-clique $C$ of $G$ containing $x_1$. There are two cases to examine:
First, if $x_d \notin C$, then by construction there exists some $e_i$ with $1 \leq i \leq m$ such that $e_i \subseteq C$. In the case where $x_d \in C$, the remaining $t-1$ vertices of $C$ must form one of the $e_i$, again yielding $e_i \subseteq C$. This establishes the equality $\Omega_m(G) = \mathcal{CH}_t(G\setminus x_1)$.
\end{proof}

The following example demonstrates that the converse of Theorem~\ref{linear-rs} does not hold.

\begin{example}
    Consider the graph $G$ with vertex set $V(G) = \{x_1, x_2, x_3, x_4, x_5, x_6\}$ and edge set
    $E(G) = \big\{\{x_1, x_2\}, \{x_2, x_3\}, \{x_3, x_4\}, \{x_4, x_1\}, \{x_1, x_5\}, \{x_1, x_6\}, \{x_5, x_6\}\big\}.$
    Observe that $G$ is \emph{not} chordal, as it contains the cycle $x_1, x_2, x_3, x_4, x_1$ without a chord.
However, the edge set of the complement of the $3$-clique complex $\overline{\CH_3(G)}$ satisfies
    $E(\overline{\CH_3(G)}) = E(\CH_3(K_6) \setminus \{x_1, x_5, x_6\}),$
    where $K_6$ denotes the complete graph on six vertices.
Despite $G$ not being chordal, the ideal $I(\overline{\CH_3(G)})$ admits a $3$-linear resolution.
\end{example}

 \vspace*{2mm}
\noindent
\textbf{Acknowledgments.}
We thank P.~Deshpande and A.~Singh for organizing the NCM workshop
"Cohen-Macaulay Simplicial Complexes in Graph Theory" (2023) at CMI, which
introduced us to this area and led to several valuable discussions. Both authors
acknowledge support from the Science and Engineering Research Board (SERB), and
the second author additionally acknowledges support from the National Board for
Higher Mathematics (NBHM).

\vspace*{2mm}
\noindent
\textbf{Data availability statement.} Data sharing is not applicable to this article as no datasets were generated or analyzed during the current study.

\bibliographystyle{abbrv}
\bibliography{refs} 
\end{document}